\documentclass{amsart}

\usepackage{amssymb}
\usepackage{amsthm}
\usepackage{amsmath}
\usepackage{graphicx}
\usepackage{pifont}
\usepackage{txfonts}
\usepackage[all]{xy}
\usepackage{color}

\theoremstyle{plain}
\newtheorem{theorem}{Theorem}[section]
\newtheorem{proposition}[theorem]{Proposition}
\newtheorem{corollary}[theorem]{Corollary}
\newtheorem{lemma}[theorem]{Lemma}

\newtheorem*{Theorem}{Theorem}

\newtheorem*{conjecture*}{Conjecture}
\theoremstyle{definition}
\newtheorem{definition}[theorem]{Definition}
\newtheorem*{definition*}{Definition}
\newtheorem{example}[theorem]{Example}
\newtheorem*{example*}{Example}
\newtheorem*{notation*}{Notation}
\newtheorem*{notation-conv*}{Notation and convention}
\newtheorem*{convention*}{Convention}
\theoremstyle{remark}
\newtheorem{remark}[theorem]{Remark}
\newtheorem*{remark*}{Remark}

\newcommand{\N}{\mathbb{N}}
\newcommand{\Z}{\mathbb{Z}}

\newcommand{\C}{\mathbb{C}}

\newcommand{\SL}[1][2]{\mathrm{SL}_{#1}(\C)}

\newcommand{\trace}{{\rm tr}\,}
\newcommand{\I}{\mathbf{1}}
\newcommand{\bm}[1]{\mbox{\boldmath{$#1$}}}

\newcommand{\bnd}[1]{\partial_{#1}}
\newcommand{\im}{\mathop{\mathrm{Im}}\nolimits}

\newcommand{\bd}{\partial}
\newcommand{\lk}{\mathop{\mathrm{\ell k}}\nolimits}

\newcommand{\Tor}[2]{\mathop{\mathrm{Tor}}\nolimits (#1;#2)}

\newcommand{\ie}{i.e.,\,}
\begin{document}


\title[]{
  The asymptotics of the higher dimensional Reidemeister torsion for exceptional surgeries
  along twist knots
}

\author{Anh T.~Tran \and Yoshikazu Yamaguchi}

\address{Department of Mathematical Sciences, 
  The University of Texas at Dallas, 
  Richardson, TX 75080, USA}
\email{att140830@utdallas.edu}

\address{Department of Mathematics,
  Akita University,
  1-1 Tegata-Gakuenmachi, Akita, 010-8502, Japan}
\email{shouji@math.akita-u.ac.jp}


\keywords{Reidemeister torsion, graph manifold, asymptotic behavior, exceptional surgery}
\subjclass[2010]{57M27, 57M50}

\begin{abstract}
  We determine the asymptotic behavior of the higher dimensional Reidemeister torsion for
  the graph manifolds obtained by exceptional surgeries along twist knots.
  We show that all irreducible $\SL$-representations of the graph manifold
  are induced by irreducible metabelian representations of the twist knot group.
  We also give the set of the limits of the leading coefficients
  in the higher dimensional Reidemeister torsion explicitly.
\end{abstract}


\maketitle

\section{Introduction}
The purpose of this paper is to observe the asymptotic behavior of
the higher dimensional Reidemeister torsion for graph manifolds.
In particular, we are interested in graph manifolds
whose $\SL$-representations of the fundamental groups are described by
certain subsets of the $\SL$-representations of hyperbolic knot groups.

A closed orientable irreducible 3-manifold $M$ is called a graph manifold
if there exists disjoint incompressible tori $T^2_1, \ldots, T^2_k$ in $M$
such that each component of $M \setminus (T^2_1 \cup \ldots \cup T^2_k)$
is a Seifert fibered space and the whole space $M$ does not admit any Seifert fibration.
It has been shown in~\cite{Yamaguchi:asymptoticsRtorsion} that
the higher dimensional Reidemeister torsion for a Seifert fibered space 
grows exponentially and its logarithm has 
the same order as the dimension of representations.
It is natural to expect that 
we have the same growth order in the case of a graph manifold. 
In this paper, 
we determine the growth order and the limit of the leading coefficient in the sequence given by the logarithm of
the higher dimensional Reidemeister torsion for certain graph manifolds. 
We will see the difference
in the limit of the leading coefficient between our graph manifolds
and the Seifert fibered spaces studied in~\cite{Yamaguchi:asymptoticsRtorsion}. 

In the study of exceptional surgeries along a hyperbolic knot, the problem of finding incompressible tori that cut the resulting manifold into Seifert fibered spaces has been investigated.
For example there exists a complete list~\cite{BrittenhamWu}
of exceptional surgeries along two--bridge knots. 
The torus decomposition of the resulting graph manifolds
is also given in~\cite{Patton, ClayTeragaito, teragaito13:LO_TwistKnots}.

When a manifold is obtained by a surgery along a knot,
its fundamental group is given by a quotient group of the knot group.
Therefore we can pull--back $\SL$-representations from the fundamental group of the resulting manifold
to the knot group (for the details, see Section~\ref{sec:representations_twist_knot}).
The $\SL$-representation space of a hyperbolic knot group
can be regarded as a parameter space for deformations of the hyperbolic structure
of the knot exterior.
Since exceptional surgeries along a hyperbolic knot yield non--hyperbolic manifolds,
the resulting manifolds induce $\SL$-representations of the hyperbolic knot group
which correspond to degenerate hyperbolic structures.
We are also motivated to see the asymptotic behavior
of the higher dimensional Reidemeister torsion 
when we choose an $\SL$-representation for a hyperbolic $3$-manifold
which is different from the holonomy representation. Here the holonomy representation is an $\SL$-representation 
corresponding to the complete hyperbolic structure.
We wish to investigate the asymptotic behavior of the higher dimensional Reidemeister torsion
for degenerate hyperbolic structures through the $\SL$-representations induced by
an exceptional surgery.

For our purpose, 
we choose hyperbolic twist knots (see Fig.~\ref{fig:twistknot}) with $4$-surgeries.
According to the torus decomposition in~\cite{Patton},
in the set of exceptional surgeries along two--bridge knots,
only $4$-surgeries along hyperbolic twist knots yield
graph manifolds consisting of two Seifert fibered spaces
which include a torus knot exterior.
More precisely, $4$-surgery along 
a twist knot $K_n$ illustrated in Fig.~\ref{fig:twistknot}
yields the graph manifold $M$ consisting of the torus knot exterior of type $(2, 2n+1)$,
which will be denoted by $T(2, 2n+1)$, and the twisted $I$-bundle over the Klein bottle.
We consider the asymptotic behavior of the higher dimensional Reidemeister torsion for $M$.
When we choose a homomorphism $\bar\rho$
from $\pi_1(M)$ into $\SL$, we also have a sequence of homomorphisms $\sigma_{2N} \circ \bar\rho$
from $\pi_1(M)$ into $\SL[2N]$ by the composition with the irreducible representations 
$\sigma_{2N}$ of $\SL$ into $\SL[2N]$.
Our main theorem is stated as follows.
\begin{Theorem}[Theorem~\ref{thm:main_theorem} and Corollary~\ref{cor:set_limits}]
  The growth of $\log |\Tor{M}{\sigma_{2N} \circ \bar\rho}|$
  has the same order as $2N$ for every irreducible $\SL$-representation $\bar\rho$ of
  $\pi_1(M)$.
  The limits of the leading coefficients are expressed as
  \begin{align*}
    &\left\{\left.
    \lim_{N \to \infty} \frac{\log |\Tor{M}{\sigma_{2N} \circ \bar\rho}|}{2N} \,\right|\,
    \hbox{$\bar\rho$ is irreducible}\right\}\\
    &=
    \left\{\left.
    \frac{1}{p_k} (\log|\Delta_{T(2, 2n+1)}(-1)| - \log 2) \,\right|\,
    p_k > 1, \hbox{$p_k$ is a divisor of $|\Delta_{K_n}(-1)|$}
    \right\}
  \end{align*}
  where $\Delta_K(t)$ is the Alexander polynomial of a knot $K$.

  In particular, the minimum in the limits of the leading coefficients is given by
  \[ \frac{1}{|\Delta_{K_n}(-1)|} (\log|\Delta_{T(2, 2n+1)}(-1)| - \log 2).\]
\end{Theorem}

We will show our main theorem by the following procedures.
First we will see that all irreducible $\SL$-representations $\bar\rho$ of $\pi_1(M)$
are induced by irreducible metabelian representations $\rho$ of a twist knot group $\pi_1(E_{K_n})$.
Here $E_{K_n}$ is the knot exterior of a twist knot $K_n$.
Concerning the decomposition of $M$ as the union of $E_{T(2, 2n+1)}$ and
the twist $I$-bundle $N(Kb)$ over the Klein bottle $Kb$, 
the restriction of $\bar\rho$ to $\pi_1(E_{T(2, 2n+1)})$ is abelian.
On the other hand, the restriction to $\pi_1(N(Kb))$ is irreducible.
We can also compute the Reidemeister torsion for $M$ and $\bar\rho$
by the product of the Reidemeister torsions for $E_{T(2, 2n+1)}$ and $N(Kb)$
in the JSJ decomposition of $M$.
We will obtain the limits of the leading coefficients in our main theorem from the observation about
the asymptotic behavior of the Reidemeister torsion for 
the torus knot exterior $E_{T(2, 2n+1)}$ and abelian representations
given by the restrictions of $\bar\rho$. We  remark that, since $|\Delta_K(-1)|$ is always odd,
  these limits differ from 
  the limit of the leading coefficient 
  for the exterior of the torus knot $T(2, 2n+1)$ and an {\it irreducible} $\SL$-representation
  in~\cite{Yamaguchi:asymptoticsRtorsion}, 
  which is given by 
  $(1-1/2-1/q')\log2$ with a divisor $q' ( > 1)$ of $2n+1$.
  The maximum of $(1-1/2-1/q')$ is equal to $-\chi$ where 
  $\chi$ is the Euler characteristic of
  the base orbifold in the Seifert fibration of the exterior $T(2, 2n+1)$.

From the viewpoint of hyperbolic structures, 
$4$-surgery along a hyperbolic twist knot yields
degenerate hyperbolic structures of the twist knot exterior.
In this paper, we see that such degenerate hyperbolic structures are given by
irreducible metabelian representations in the $\SL$-representation space of a twist knot group.
The above Theorem
(Theorem~\ref{thm:main_theorem} and Corollary~\ref{cor:set_limits}) 
and the results in~\cite{FerrerPorti:HigherDimReidemeister, porti:survey_Rtorsion}
imply that, in the case of a hyperbolic twist knot exterior, 
the growth order of the higher dimensional Reidemeister torsion  
for any irreducible metabelian representation 
decreases from that for the holonomy representation.
Note that the Reidemeister torsion under our convention 
is the inverse of that of~\cite{FerrerPorti:HigherDimReidemeister}
(for more details, see~\cite{porti:survey_Rtorsion}). 
We will observe that this degeneration occurs for any knot 
in the subsequent paper~\cite{tranYam:higherdimTAPmetabelian}.
In other words, we will observe that the growth order of
the higher dimensional Reidemeister torsion 
for any irreducible metabelian representation of a hyperbolic knot group
is less than that for the holonomy representation.

\section{Preliminaries}
\subsection{The higher dimensional Reidemeister torsion}
For the Reidemeister torsion,
we follow the notation and definition used in~\cite{Yamaguchi:asymptoticsRtorsion}.
For the details and related topics, we refer to the survey articles~\cite{Milnor:1966, porti:survey_Rtorsion} by
J.~Milnor and J.~Porti or the book~\cite{Turaev:2000} by V.~Turaev.
We need a homomorphism from the fundamental group into $\SL$ 
to observe the Reidemeister torsion for a manifold. 
Throughout this paper, a homomorphism from a group $H$ into a linear group $G$
will be referred to as {\it a $G$-representation} of $H$.
The symbol $\sigma_n$ denotes the right action of $\SL$ on the vector space
$V_{n}$, consisting of homogeneous polynomials $p(x, y)$ of degree $n-1$,
defined as 
\[
\sigma_n (A) \cdot p(x, y) = p(x', y')\quad \hbox{where}\,
\begin{pmatrix} x' \\ y' \end{pmatrix}
= A^{-1}\begin{pmatrix} x \\ y \end{pmatrix}.
\]
It is known that this action induces a homomorphism from $\SL$ into $\SL[n]$,
which is referred as {\it the $n$-dimensional irreducible representation of $\SL$}. 
For the simplicity,
we use the same symbol $\sigma_n$ for the $n$-dimensional irreducible representation of $\SL$.
We mainly use the $2N$-dimensional irreducible representation $\sigma_{2N}$.
If $A \in \SL$ has the eigenvalues $\xi^{\pm 1}$,
then $\sigma_{2N}(A)$ has the eigenvalues $\xi^{\pm 1}$, $\xi^{\pm 3}, \ldots, \xi^{\pm (2N-1)}$.
This is due to the following action
\[\sigma_{2N}(A) \cdot (x^{2N-1 - i}y^i) = \xi^{-2N+1+2i} (x^{2N-1 - i}y^i)
\quad\hbox{for}\quad
A=
\begin{pmatrix}
  \xi & 0 \\
  0 & \xi^{-1}
\end{pmatrix}
\]
on the standard basis $\{x^{2N-1}, x^{2N-2}y, \ldots, xy^{2N-2}, y^{2N-1}\}$ of $V_{2N}$.
\begin{definition}
  Let $W$ be a finite CW-complex and $\rho$ be an $\SL$-representation of $\pi_1(W)$.
  The twisted chain complex $C_*(W;V_n)$ with the coefficient $V_{n}$ is defined as
  a chain complex which consists of
  \[ C_i(W;V_n) = V_n \otimes_{\sigma_n \circ \rho} C_i(\widetilde W;\Z) \]
  where $\widetilde W$ is the universal cover of $W$
  and $C_i(\widetilde W;\Z)$ is a left $\Z[\pi_1(W)]$-module.
\end{definition}
We assume that each twisted chain module $C_*(W;V_n)$ is equipped with 
a basis $\bm{c}^i$ given by $v_j \otimes \tilde e^i_{j'}$ where $v_j$ is a vector in a basis of $V_n$
and $\tilde e^i_{j'}$ is a lift of an $i$-dimensional cell $e^i_{j'}$ in $W$.
\begin{definition}
  Suppose that  the twisted chain complex $C_*(W;V_n)$ is acyclic,
  \ie $\im \bnd{i} = \ker\bnd{i-1}$ for all $i$.
  Each chain module $C_i(W;V_n)$ has the following decomposition:
  \[C_i(W;V_n) = \bnd{i+1}\tilde B_{i+1} \oplus \tilde B_{i}\] 
  where $\tilde B_{i}$ is a lift of $\im \bnd{i}$.
  Then we will denote by $\Tor{W}{\sigma_{n} \circ \rho}$ the $n$-dimensional Reidemeister torsion for
  $W$ and $\rho$, which is given by the following alternating product:
  \begin{equation}
    \label{eqn:def_torsion}
    \prod_{i \geq 0} \det \left(\bnd{i+1}\tilde{\bm{b}}^{i+1} \cup \tilde{\bm{b}}^i / \bm{c}^i\right)^{(-1)^{i+1}}
  \end{equation}
  where $\tilde{\bm{b}}^i$ is a basis of $\tilde B_i$, $\bm{c}^i$ is an equipped basis of $C_i(W;V_n)$
  and $(\bnd{i+1}\tilde{\bm{b}}^{i+1} \cup \tilde{\bm{b}}^i / \bm{c}^i)$ is the base change matrix
  from $\bm{c}^i$ to $\bnd{i+1}\tilde{\bm{b}}^{i+1} \cup \tilde{\bm{b}}^i$.
\end{definition}

There are several choices in the definition of the $n$-dimensional Reidemeister torsion.
Let us mention the well-definedness of the Reidemeister torsion without proofs.
We refer to~\cite{porti:survey_Rtorsion, Yamaguchi:asymptoticsRtorsion}
for the details.
\begin{remark}
  The alternating product~\eqref{eqn:def_torsion} is independent of a choice of a lift of $\im \bnd{i}$.
  If the Euler characteristic of $W$ is zero, then $\Tor{W}{\sigma_{n} \circ \rho}$ is also independent
  of a choice of a basis of $V_n$.
  It is known that $\Tor{W}{\sigma_{n} \circ \rho}$ does not depend on the ordering and orientation
  of cells in $\bm{c}^i$
  when $n$ is even.
  This is a reason why we restrict our attention to $2N$-dimensional ones.
\end{remark}

We give an example of $2N$-dimensional Reidemeister torsion which will be needed in this paper.
\begin{figure}[ht]
  \centering
  \includegraphics[scale=.6]{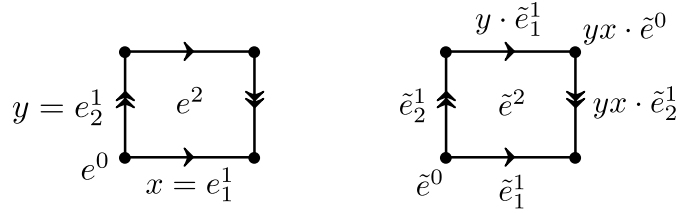}
  \caption{a cell decomposition of $Kb$ (left) and
    a lift to $\widetilde{Kb}$ (right)} 
  \label{fig:Kb}
\end{figure}
\begin{example}
  \label{example:Kb}
  Suppose that the Klein bottle $Kb$ is decomposed as in Fig.~\ref{fig:Kb} and
  $\rho$ is an $\SL$-representation of $\pi_1(Kb)$.
  The fundamental group has the presentation $\pi_1(Kb) = \langle x, y \,|\, yx = xy^{-1}\rangle$.
  The twisted chain complex $C_*(Kb;V_n)$ is expressed as
  \begin{gather*}
  0 \to
  C_2(Kb;V_n)=V_n \xrightarrow{\bnd{2}} C_1(Kb;V_n) = V_n \oplus V_n \xrightarrow{\bnd{1}} C_0(Kb;V_n)=V_n
  \to 0 \\
  \bnd{2} = \begin{pmatrix} \I - Y \\ - XY - \I \end{pmatrix}, \quad
  \bnd{1} = \begin{pmatrix} X - \I & Y - \I \end{pmatrix} 
  \end{gather*}
  where $X = \sigma_n \circ \rho (x)$ and $Y = \sigma_n \circ \rho (y)$.
  By the relation $x^{-1}yx=y^{-1}$,
  the $\SL$-representation $\rho$ is classified into the following three cases,
  up to conjugation:
  \begin{enumerate}
  \item \label{item:abelian}
    $\rho(y) = \pm \I$ and $\rho(x)$ is arbitrary,
  \item \label{item:irreducible}
    $\rho(y)
    = \begin{pmatrix}
      \eta & 0 \\
      0 & \eta^{-1}
    \end{pmatrix}\, (\eta \not = \pm1)$
    and
    $\rho(x)
    = \begin{pmatrix}
      0 & -1 \\
      1 & 0
    \end{pmatrix}$,
  \item \label{item:reducible_nonabelian}
    $\rho(y)
    = \begin{pmatrix}
      \pm 1 & \omega \\
      0 & \pm 1
    \end{pmatrix}\, (\omega \not = 0)$
    and
    $\rho(x)
    = \begin{pmatrix}
      \pm \sqrt{-1} & \omega' \\
      0 & \mp \sqrt{-1}
    \end{pmatrix}$.
  \end{enumerate}

  We can express the $2N$-dimensional Reidemeister torsion $\Tor{Kb}{\sigma_{2N} \circ \rho}$ as
  \begin{equation}
    \label{eqn:Rtorsion_Kb}
  \Tor{Kb}{\sigma_{2N} \circ \rho}
  =\begin{cases}
  \displaystyle \frac{\det(\I-Y)}{\det(Y-\I)} & (\det(Y-\I) \not = 0) \medskip \\ 
  \displaystyle \frac{\det(-XY-\I)}{\det(X-\I)} & (\det(Y-\I) = 0 ).
  \end{cases}
  \end{equation}
  Note that the left edge in Fig.~\ref{fig:Kb} is moved to the right one by the covering transformation of $yx$
  since the starting point of the left edge is moved to that of the right edge by $yx \in \pi_1(Kb)$. 
\end{example}

We will use the following gluing formula
of the $2N$-dimensional Reidemeister torsion.
This is an application of {\it the Multiplicativity property}
of the Reidemeister torsion to a torus decomposition of a $3$-manifold.
In the case of the $2N$-dimensional Reidemeister torsion, 
we can determine the sign in the gluing formula easily.
For the details on applying the Multiplicativity property to a decomposition along a torus,
we refer to~\cite[Subsection~2.3 and Section~3]{Yamaguchi:asymptoticsRtorsion}
and the references given there.
\begin{lemma}[Consequence of the Multiplicativity property for a decomposition along a torus]
  \label{lemma:mult_prop}
  Suppose that a compact $3$-manifold $M$ is the union $M_1 \cup_{T^2} M_2$ and
  each $M_i$ is given a CW-structure such that 
  both of them induce the same CW-structure of $T^2$.
  If an $\SL$-representation $\rho$ of $\pi_1(M)$
  induces the acyclic complexes $C_*(M_1;V_{2N})$, $C_*(M_2;V_{2N})$ and $C_*(T^2;V_{2N})$,
  then the twisted chain complex $C_*(M;V_{2N})$ defined by $\rho$
  is also acyclic and the $2N$-dimensional Reidemeister
  torsion $\Tor{M}{\sigma_{2N} \circ \rho}$ is expressed as
  \[
  \Tor{M}{\sigma_{2N} \circ \rho}
  = \Tor{M_1}{\sigma_{2N} \circ \rho}\Tor{M_2}{\sigma_{2N} \circ \rho}.
  \]
\end{lemma}

\begin{remark}
  Usually we have the equality that
  \[
  \Tor{M}{\sigma_{2N} \circ \rho}\Tor{T^2}{\sigma_{2N} \circ \rho}
  = \Tor{M_1}{\sigma_{2N} \circ \rho}\Tor{M_2}{\sigma_{2N} \circ \rho}
  \]
  as a consequence of the Multiplicativity property.
  It is known that $\Tor{T^2}{\sigma_{2N} \circ \rho} = 1$ if it is defined.
\end{remark}

\subsection{$\SL$-representations of twist knot groups}
\label{sec:representations_twist_knot}
We review several results concerning $\SL$-representations of the fundamental groups
of our graph manifolds.
We write $E_K$ for the knot exterior of a knot $K$,
which is obtained by removing an open tubular neighbourhood of $K$ from $S^3$.
We mainly consider the $n$-twist knot $K_n$, illustrated in Figure~\ref{fig:twistknot}.
The horizontal twists are right-handed if $n$ is positive, left-handed if $n$ is negative.
\begin{figure}[ht]
  \centering
  \includegraphics[scale=.6]{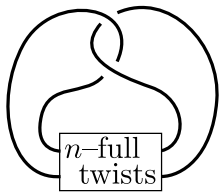}
  \caption{a diagram of $K_n$}
  \label{fig:twistknot}
\end{figure}
Under our convention, the $1$-twist knot $K_1$ is the figure-eight knot.

It is known that $K_n$ is a hyperbolic knot and that
$4$-surgery along $K_n$ yields a graph manifold $M$ when $n \not = 0, -1$.
The fundamental group $\pi_1(M)$ has the following presentation. 
\begin{proposition}[Proposition~$2.2$ in~\cite{teragaito13:LO_TwistKnots}]
  \label{prop:pi_1_M}
  The graph manifold $M$ consists of a torus knot exterior $E_{T(2, 2n+1)}$ and the twisted $I$-bundle over the Klein bottle.
  The fundamental group has a presentation:
  \begin{equation}
    \label{eqn:pres_pi_1_M}
  \pi_1(M)
  = \langle a, b, x, y\,|\,a^2 = b^{2n+1}, x^{-1} y x = y^{-1}, \mu = y^{-1}, h=y^{-1}x^2 \rangle
  \end{equation}
  where $\mu = b^{-n} a$ and $h$ correspond to
  a meridian and a regular fiber of the torus knot exterior (with the Seifert fibration),
  respectively.
\end{proposition}

Since $\pi_1(M)$ is isomorphic to the quotient group
$\pi_1(E_{K_n}) / \langle\!\langle m^4\ell \rangle\!\rangle$
where $m$ and $\ell$ are a meridian and a preferred longitude on $\bd E_{K_n}$,
Proposition~\ref{prop:pi_1_M} shows that the quotient
$\pi_1(E_{K_n}) / \langle\!\langle m^4\ell \rangle\!\rangle$
is expressed as~\eqref{eqn:pres_pi_1_M}.
We denote by $\bar \rho$ the induced homomorphism from $\pi_1(M)$ into $\SL$:
\[\xymatrix{
  \pi_1(E_{K_n}) \ar[r]^{\rho} \ar[d]& \SL \\
  \pi_1(M)\ar[ur]_{\bar\rho} &
}\]

\begin{definition}
  An $\SL$-representation $\rho$ of a group $H$ is referred as being {\it irreducible}
  if the invariant subspaces of $\C^2$ under the action of $\rho(H)$ are only $\{\bm{0}\}$ and $\C^2$.
  An $\SL$-representation $\rho$ is called {\it reducible} if it is not irreducible.
  We also call $\rho$ {\it abelian} if the image $\rho(H)$ is an abelian subgroup in $\SL$.
\end{definition}

\begin{remark}
  \label{rem:irred_rhobar}
  The image of $\pi_1(E_{K_n})$ by $\rho$ coincides with that of $\pi_1(M)$ by $\bar\rho$.
  Hence $\bar\rho$ is irreducible if and only if $\rho$ is irreducible.
\end{remark}

\begin{remark}
  We have seen the classification of $\SL$-representations of $\pi_1(Kb)$
  in Example~\ref{example:Kb}.
  The case~\eqref{item:abelian} gives abelian representations, 
  the case~\eqref{item:irreducible} gives irreducible ones and
  the case~\eqref{item:reducible_nonabelian} gives reducible and non--abelian ones.
\end{remark}

\begin{definition}
  We write $R(X)$ for the set of homomorphisms from $\pi_1(X)$ into $\SL$.
  We call $R(X)$ the $\SL$-representation space of $\pi_1(X)$.
  The symbol $R^{\mathrm{irr}}(X)$ denotes the subset of irreducible representations in $R(X)$.
\end{definition}
The pull-back by the quotient induces an inclusion from $R(M)$ into $R(E_{K_n})$.
We can regard the representation space $R(M)$ as a subset in $R(E_{K_n})$.
From this viewpoint, $R(M)$ is expressed as
\[
R(M) = \{\rho \in R(E_{K_n}) \,|\, \rho(m^4\ell) = \I \}.
\]

\begin{lemma}
  \label{lemma:rep_sp_M}
  Every irreducible metabelian representation of $\pi_1(E_{K_n})$ is contained in $R(M)$.
\end{lemma}
\begin{proof}
It was shown in~\cite[Proposition~1.1]{Nagasato07:Finite_of_section}
that any irreducible metabelian representation of a knot group sends a preferred longitude to $\I$
and a meridian to a trace-free matrix in $\SL$, which has the order of $4$.
\end{proof}

For any knot $K$,
we can express the set of irreducible metabelian representations as 
the union of $(|\Delta_K(-1)| - 1)/2$ conjugacy classes
where $\Delta_K(t)$ is the Alexander polynomial of $K$.
If $K$ is a twist knot $K_n$,
then we have the following representatives of conjugacy classes.
Here we suppose that $\pi_1(E_{K_n})$ has a presentation
$\pi_1(E_{K_n}) = \langle \alpha, \beta \,|\, \omega^n \alpha = \beta \omega^n \rangle$
where $\alpha$, $\beta$ are meridians and $\omega = \beta \alpha^{-1} \beta^{-1} \alpha$.
A twist knot $K_n$ has $(|4n+1|-1)/2$ conjugacy classes
since its  Alexander polynomial is given by $-nt^{2}+(2n+1)t-n$.
\begin{proposition}[Theorem~3 in~\cite{NagasatoYamaguchi} for $K_n$ ]
  \label{prop:rho_k}
  The set of irreducible metabelian representations of $\pi_1(E_{K_n})$ consists of 
  $(|4n+1|-1)/2$ conjugacy classes. The representatives are given by
  the following $\rho_k$ ($k=1, \ldots, (|4n+1|-1)/2$):
  \[
  \rho_k(\alpha) =
  \begin{pmatrix}
    \sqrt{-1} & -\sqrt{-1} \\
    0 & -\sqrt{-1}
  \end{pmatrix},\quad
  \rho_k(\beta) =
  \begin{pmatrix}
    \sqrt{-1} & 0 \\
    -u_k\sqrt{-1} & -\sqrt{-1}
  \end{pmatrix},\quad
  u_k=-4\sin^2\frac{k\pi}{4n+1}.
  \]
\end{proposition}

\section{Representation spaces for resulting graph manifolds}
Let $M$ be the graph manifold obtained by $4$-surgery along a hyperbolic twist knot $K_n$.
\subsection{$R(M)$ as a subspace of $R(E_{K_n})$}
We determine the $\SL$-representation space $R(M)$ as a subset in $R(K_n)$.
\begin{proposition}
  \label{prop:rep_sp_M}
  Every irreducible representation of $\pi_1(M)$ into $\SL$ is induced
  by an irreducible metabelian one of $\pi_1(E_{K_n})$, \ie
  \[R^{\mathrm{irr}}(M) = \{\rho \in R(E_{K_n}) \,|\, \hbox{$\rho$ is irreducible metabelian}\}.\]
\end{proposition}
\begin{proof}
  By Lemma~\ref{lemma:rep_sp_M},
  it is sufficient to show that if any irreducible representation $\rho$ of $\pi_1(E_{K_n})$ factors through
  the quotient group $\pi_1(E_{K_n}) / \langle\!\langle m^4\ell \rangle\!\rangle$, then $\rho$ is metabelian.
  When $\mathcal{M}^{\pm 1}$ denote the eigenvalues of $\rho(m)$, 
  the trace $\mathcal{M} + \mathcal{M}^{-1}$ of $\rho(m)$ must be zero
  by Lemmas~\ref{lemma:Apoly_slope} and~\ref{lemma:Apoly_explicit} below.
  Since $K_n$ is a two--bridge knot,
  it follows from~\cite[Lemma~23]{NagasatoYamaguchi} that $\rho$ must be a metabelian representation.
\end{proof}
\begin{lemma}
  \label{lemma:Apoly_slope}
  If an irreducible representation $\rho \in R(E_{K_n})$ factors through $\pi_1(M)$, 
  then the eigenvalue $\mathcal{M}$ satisfies that
  $A_{K_n}(\mathcal{M}^{-4}, \mathcal{M})=0$ where
  $A_{K_n}(\mathcal{L}, \mathcal{M})$ is the A-polynomial of $K_n$.
\end{lemma}
\begin{proof}
  The A-polynomial $A_{K_n}(\mathcal{L}, \mathcal{M})$ gives
  the defining equation of $R(\bd E_{K_n})$.
  Since the peripheral group $\pi_1(\bd E_{K_n})$ is an abelian group, 
  we can assume that the images of $\rho(m)$ and $\rho(\ell)$ are upper triangular matrices
  whose diagonal entries are $\mathcal{M}^{\pm 1}$ and $\mathcal{L}^{\pm 1}$ respectively.
  Then we can rewrite the constraint that $\rho(m^4\ell) = \I$ as 
  $\mathcal{L}=\mathcal{M}^{-4}$. 
  The lemma follows.
\end{proof}
\begin{lemma}
  \label{lemma:Apoly_explicit}
  The A-polynomial of $K_n$ for $\mathcal{L}=\mathcal{M}^{-4}$ is expressed as
  \[
  A_{K_n}(\mathcal{M}^{-4}, \mathcal{M})
    = \begin{cases}
    \mathcal{M}^{-8n}(\mathcal{M}+\mathcal{M}^{-1})^{2n} & (n>0)\\
    \mathcal{M}^{-8|n|+3}(\mathcal{M}+\mathcal{M}^{-1})^{2|n|-1} & (n<0).
  \end{cases}
  \]
\end{lemma}
\begin{proof}
  Since the knot $K_n$ is the mirror image of $J(2, -2n)$ in~\cite{HosteShanahan},
  the A-polynomial $A_{K_n}(\mathcal{L}, \mathcal{M})$ coincides with
  $A_{J(2, -2n)}(\mathcal{L}, \mathcal{M}^{-1})$.
  Hence we have that
  \[A_{K_n}(\mathcal{M}^{-4}, \mathcal{M}) = A_{J(2, -2n)}(\mathcal{M}^{-4}, \mathcal{M}^{-1}).\]
  By induction and the recursive formula in~\cite[Theorem~1]{HosteShanahan},
  one can show that 
  \[
  A_{J(2, 2n)}(\mathcal{M}^{-4}, \mathcal{M}^{-1}) =
  \begin{cases}
    \mathcal{M}^{-8n+3}(\mathcal{M}+\mathcal{M}^{-1})^{2n-1} & (n>0) \\
    \mathcal{M}^{-8|n|}(\mathcal{M}+\mathcal{M}^{-1})^{2|n|} & (n<0).
    \end{cases}
  \]
  The lemma then follows.
\end{proof}

\subsection{The restrictions to Seifert pieces}
We will see the restriction of $\bar\rho \in R^{\mathrm{irr}}(M)$ to
the fundamental group of each Seifert piece.
Recall that the graph manifold $M$ is the union the torus knot exterior $E_{T(2, 2n+1)}$ and
the twisted $I$-bundle $N(Kb)$ over the Klein bottle $Kb$.
\begin{proposition}
  \label{prop:torusknot_abelian}
  For every $\bar \rho \in R^{\mathrm{irr}}(M)$, the restriction of $\bar \rho$ to $\pi_1(E_{T(2, 2n+1)})$
  is abelian. 
\end{proposition}
\begin{proof}
  It was shown by~\cite[Theorem~1.2]{Teragaito03:ToroidalII} that
  a twist knot $K_n$ bounds a once-punctured Klein bottle whose boundary slope is $4$.
  We can think of loops in $E_{T(2, 2n+1)}$ as loops
  outside a non--orientable spanning surface of $K_n$ in $E_{K_n}$.
  A loop $\gamma$ outside a non--orientable spanning surface of $K_n$ has an even linking number with $K_n$.
  When we express $\gamma \in \pi_1(E_{K_n})$ as
  $\gamma = m^{\lk(\gamma, K_n)} (m^{-\lk(\gamma, K_n)} \gamma)$,
  we have the even integer $\lk(\gamma, K_n)$ and the commutator $m^{-\lk(\gamma, K_n)} \gamma$.
  By Proposition~\ref{prop:rep_sp_M},
  one can see that $\bar \rho$ is induced by an irreducible metabelian representation $\rho$
  of $\pi_1(E_{K_n})$.
  Since $\rho$ sends $m^2$ and the commutator subgroup to $-\I$ and an abelian subgroup respectively,
  the image of $\pi_1(E_{T(2, 2n+1)})$ by $\rho$ is contained in the abelian subgroup.
\end{proof}

In general, any abelian representation of a knot group $\pi_1(E_K)$ is determined, up to conjugation,
by the eigenvalues of the matrix corresponding to a meridian.
This follows from the fact that any abelian representation factors through 
the abelianization $\pi_1(E_K) \to H_1(E_K;\Z)$ and
$H_1(E_K;\Z)$ is generated by the homology class of a meridian.

\begin{lemma}
  \label{lemma:relation_eigenvalue}
  Every $\bar\rho \in R^{\mathrm{irr}}(M)$ is determined by the eigenvalues of $\bar\rho(\mu)$
  up to conjugation.
\end{lemma}

Furthermore the set of eigenvalues is determined as follows.
\begin{proposition}
  \label{prop:set_eigenvalues}
  Suppose that $\rho_k \in R(K_n)$ is an irreducible metabelian representation
  in Proposition~\ref{prop:rho_k} and
  $\mu$ is a meridian of the torus knot in the presentation~\eqref{eqn:pres_pi_1_M}.
  Let $\xi_k^{\pm 1}$ be the eigenvalues of $\bar\rho_k(\mu)$. Then 
  the set $\{\xi_k^{\pm 1} \,|\, k=1, \ldots, (|4n+1|-1)/2\}$ is given by
  $\{e^{\pm \theta \sqrt{-1}} \,|\, \theta = \pi (2j-1)/|4n+1|, j=1,\ldots,(|4n+1|-1)/2\}$.
\end{proposition}

\begin{proof}
  Let $p$ be $|4n+1|$. We regard elements of $\pi_1(E_{T(2, 2n+1)})$ as the products 
  $m^{2r}\gamma$ where $r \in \Z$ and $\gamma$ is a commutator of $\pi_1(E_{K_n})$
  as in the proof of Proposition~\ref{prop:torusknot_abelian}.
  It follows from~\cite[Proposition~2.8]{yamaguchi:twistedAlexMeta} that
  the eigenvalues of $\rho_k(\gamma)$ are $p$-th roots of unity.
  Since $\rho_k(m^2) = -\I$ and $p$ is odd,
  one can see that for the generators $a$ and $b \in \pi_1(E_{T(2, 2n+1)})$
  \[{\bar \rho_k(a)}^p = \pm \I \quad \hbox{and} \quad {\bar \rho_k(b)}^p = \pm \I.\]
  By the relation $a^2=b^{2n+1}$, we can conclude that ${\bar \rho_k(b)}^p = \I$.
  On the other hand, we can see that ${\bar\rho_k(a)}^p = -\I$
  since the image of $\pi_1(E_{T(2, 2n+1)})$ by $\bar\rho_k$ contains $-\I$ and $p$ is odd.
  Hence the eigenvalues $\xi_k^{\pm 1}$ of $\bar\rho_k(\mu) = \bar\rho_k(b^{-n}a)$ satisfy that
  $\xi_k^{\pm p} = -1$.
  We can exclude the case that $\bar\rho_k(\mu)=-\I$ by the irreducibility of $\bar\rho_k$.

  There exist at least $(|4n+1|-1)/2$ distinct pairs of eigenvalues 
  by Proposition~\ref{prop:rep_sp_M} and Lemma~\ref{lemma:relation_eigenvalue}. 
  On the other hand, there exist at most $(|4n+1|-1)/2$ distinct pairs in the set of $2p$-th roots of unity to be
  the eigenvalues $\xi_k^{\pm 1}$ of $\rho_k(\mu)$ ($k=1, \ldots, (|4n+1|-1)/2$).
  This proves Proposition~\ref{prop:set_eigenvalues}.
\end{proof}

\begin{corollary}
  \label{cor:order_mu}
  The order of $\bar\rho_k(\mu)$ is given by $2p_k$ for some divisor $p_k$ of $|\Delta_{K_n}(-1)|=|4n+1|$.
\end{corollary}

We next turn to the restriction to $\pi_1(N(Kb))$.
\begin{proposition}
  For every $\bar \rho \in R^{\mathrm{irr}}(M)$,
  the restriction of $\bar\rho$ to $\pi_1(N(Kb))$
  is irreducible.
\end{proposition}
\begin{proof}
  Note that $\trace \bar\rho(y) = \trace \bar\rho(\mu)^{-1}$
  by~\eqref{eqn:pres_pi_1_M}.
  Proposition~\ref{prop:set_eigenvalues} shows that
  $\trace \bar\rho(y) \not = \pm 2$.
  The restriction of $\bar\rho$ to $\pi_1(N(Kb))$ is an $\SL$-representation
  in the case~\eqref{item:irreducible} of Example~\ref{example:Kb}, 
  and hence is irreducible.
\end{proof}

\begin{remark}
  \label{remark:rho_y}
  By conjugation, we can assume that the $\bar\rho(a)$, $\bar\rho(b)$ and $\bar\rho(y)$ are
  diagonal matrices and $\bar\rho(x)$ is
  $\begin{pmatrix}
    0 & 1 \\
    -1 & 0
  \end{pmatrix}$
  for any $\bar\rho \in R^{\mathrm{irr}}(M)$.
  This is due to that $\bar\rho(a)$, $\bar\rho(b)$ and
  $\bar\rho(y)=\bar\rho(\mu)^{-1}$ are contained in the same maximal abelian subgroup
  in $\SL$.
\end{remark}

\section{Asymptotic behavior of Reidemeister torsion for graph manifolds}
We will consider the limit of the leading coefficient in the asymptotic behavior of
Reidemeister torsion.
We use the symbols $\xi^{\pm 1}_k$ to denote the eigenvalues of $\bar\rho_k(\mu)$.
We will compute the higher dimensional Reidemeister torsion and its asymptotic behavior
for $M$ from the decomposition of a graph manifold.

\begin{proposition}
  \label{prop:Rtorsion_2N}
  Let $\rho_k$ be an irreducible metabelian representation.
  Then the Reidemeister torsion $\Tor{M}{\sigma_{2N} \circ \bar \rho_k}$
  is expressed as
  \[
  \Tor{M}{\sigma_{2N} \circ \bar \rho_k} =
  \prod_{i=1}^{N}
  \frac{
    \Delta_{T(2, 2n+1)}(\xi^{2i-1}_k)
    \Delta_{T(2, 2n+1)}(\xi^{-2i+1}_k)
  }{
    (\xi^{2i-1}_k -1)(\xi^{-2i+1}_k -1)
  }.
  \]
\end{proposition}
\begin{proof}
  Applying Lemma~\ref{lemma:mult_prop} to the decomposition $M=E_{T(2, 2n+1)} \cup N(Kb)$,
  we have that
  \[
  \Tor{M}{\sigma_{2N} \circ \bar \rho_k}
     = \Tor{E_{T(2, 2n+1)}}{\sigma_{2N} \circ \bar \rho_k}\Tor{N(Kb)}{\sigma_{2N} \circ \bar \rho_k}.
  \]
  By Proposition~\ref{prop:torusknot_abelian} and Corollary~\ref{cor:order_mu},
  the restriction $\bar \rho_k$ to $\pi_1(E_{T(2, 2n+1)})$ is an abelian representation
  such that the matrix $\bar\rho_k(\mu)$ corresponding to a meridian has an even order.
  Our claim follows from
  Lemmas~\ref{lemma:RtorionTorus_Kb_2N} and~\ref{lemma:higher_dim_abelian} below.
\end{proof}

\begin{lemma}
  \label{lemma:RtorionTorus_Kb_2N}
  The Reidemeister torsion $\Tor{N(Kb)}{\sigma_{2N} \circ \bar \rho_k}$ is equal to $1$ for all $N$.
\end{lemma}
\begin{proof}
  By the simple homotopy equivalence, the Reidemeister torsion $\Tor{N(Kb)}{\sigma_{2N} \circ \bar \rho_k}$
  coincides with $\Tor{Kb}{\sigma_{2N} \circ \bar \rho_k}$.
  The Reidemeister torsion $\Tor{Kb}{\sigma_{2N} \circ \bar \rho_k}$ is given by Eq.~\eqref{eqn:Rtorsion_Kb}.
  The eigenvalues of $\sigma_{2N} \circ \bar \rho_k(y) = \sigma_{2N} \circ \bar\rho_k(\mu)^{-1}$
  are given by $\xi_k^{\mp (2i-1)}$ $(i=1,\ldots, N)$.
  Proposition~\ref{prop:set_eigenvalues} 
  shows that the orders of $\xi_k^{\pm 1}$ are even. Hence
  $\sigma_{2N} \circ \bar \rho_k(y)$ does not have the eigenvalue $1$ for any $N$.
  Hence Example~\ref{example:Kb} shows that 
  \[
  \Tor{N(Kb)}{\sigma_{2N} \circ \bar \rho_k}
  =\Tor{Kb}{\sigma_{2N} \circ \bar \rho_k}
  = \frac{\det(\I-Y)}{\det(Y-\I)} = 1
  \]
  for any $N$.
\end{proof}

\begin{lemma}
  \label{lemma:higher_dim_abelian}
  Let $\varphi$ be an abelian representation of a knot group $\pi_1(E_K)$
  which sends a meridian to a matrix with eigenvalues $\xi^{\pm 1}$.
  If $\xi$ is not a $(2r-1)$-root of unity for any $r \in \N$,
  then the Reidemeister torsion $\Tor{E_K}{\sigma_{2N} \circ \varphi}$ is expressed as
  \[
  \Tor{E_K}{\sigma_{2N} \circ \varphi}=
  \prod_{i=1}^{N}\frac{\Delta_K(\xi^{2i-1})}{\xi^{2i-1}-1} \frac{\Delta_K(\xi^{-2i+1})}{\xi^{-2i+1}-1}
  \]
  for all $N$.
\end{lemma}
\begin{proof}
  The Reidemeister torsion $\Tor{E_K}{\varphi}$ is given by
  $\Delta_K(\xi)\Delta_K(\xi^{-1}) / (\xi - 1)(\xi^{-1}-1)$
  (for instance, see the case of $t=1$ in~\cite[proof of Proposition~3.8]{YY2}).
  The $\SL[2N]$-representation $\sigma_{2N} \circ \varphi$ is decomposed into the direct sum
  $\oplus_{i=1}^{N} \varphi_i$ where $\varphi_i$ is an abelian representation sending a meridian
  to an $\SL$-matrix with eigenvalues $\xi^{\pm (2i-1)}$.
  For the direct sum of representations, the Reidemeister torsion is given by the product of
  those for each direct summand $\varphi_i$. This implies our claim.
\end{proof}

\begin{theorem}
  \label{thm:main_theorem}
  Let $\bar \rho_k$ be an irreducible $\SL$-representation of $\pi_1(M)$ and 
  $p_k$ be a divisor of $p=|\Delta_{K_n}(-1)|$
  such that the order of $\bar\rho_k(\mu)$ is given by $2p_k$.  
  Then the growth order of $\log|\Tor{M}{\sigma_{2N} \circ \bar \rho_k}|$ is equal to $2N$.
  Moreover the convergence of the leading coefficient is expressed as
  \begin{equation}
    \label{eqn:main_result}
    \lim_{N \to \infty} \frac{\log|\Tor{M}{\sigma_{2N} \circ \bar \rho_k}|}{2N}
    = \frac{1}{p_k} \big(\log |\Delta_{T(2, 2n+1)}(-1)| - \log 2\big).
  \end{equation} 
\end{theorem}
\begin{proof}
  It is sufficient to show that
  the left hand side of~\eqref{eqn:main_result} converges to the right hand side.
  By Proposition~\ref{prop:Rtorsion_2N},
  the left hand side of~\eqref{eqn:main_result} turns out to be
  \begin{align*}
    &\lim_{N \to \infty} \frac{\log|\Tor{M}{\sigma_{2N} \circ \bar \rho_k}|}{2N}\\
    &= \lim_{N \to \infty} 
    \frac{1}{2N}\sum_{i=1}^{N}\log|
    \Delta_{T(2, 2n+1)}(\xi^{2i-1}_k)
    \Delta_{T(2, 2n+1)}(\xi^{-(2i-1)}_k)| \\
   &\quad + \lim_{N\to \infty}
     \frac{1}{2N}\sum_{i=1}^{N}\log|
    (\xi^{2i-1}_k -1)(\xi^{-2i+1}_k-1)|^{-1}
  \end{align*}
  The eigenvalues $\xi^{\pm 1}_k$ are
  primitive $2p_k$-th roots of unity
  as in Propositions~\ref{prop:set_eigenvalues} and~\ref{cor:order_mu}.
  It follows from~\cite[Proposition~3.9]{Yamaguchi:asymptoticsRtorsion} that 
  the second term in the right hand side converges to $-(\log 2)/p_k$.
  Note that we can ignore the indeterminacy of a factor $t^{j}$ ($j \in \Z$) in the Alexander polynomial
  in the computation of the first term.
  The first term is rewritten as
  \begin{align}
    &\lim_{N \to \infty} 
    \frac{1}{2N}\sum_{i=1}^{N}\log|
    \Delta_{T(2, 2n+1)}(\xi^{2i-1}_k)
    \Delta_{T(2, 2n+1)}(\xi^{-(2i-1)}_k)| \notag\\
    &=
    \lim_{N \to \infty} 
    \frac{1}{N}\sum_{i=1}^{N}\log|
    \Delta_{T(2, 2n+1)}(\xi^{2i-1}_k)| \notag\\
    &=
    \frac{1}{p_k}\sum_{i=1}^{p_k}\log|
    \Delta_{T(2, 2n+1)}(\xi^{2i-1}_k)| \notag\\
    &=
    \frac{1}{p_k}
    \log\prod_{i=1}^{p_k}|\Delta_{T(2, 2n+1)}(\xi^{2i-1}_k)|
    \label{eqn:1st_term}
  \end{align}
  by the symmetry that $\Delta_K(t) = t^j\Delta_K(t^{-1})$ $(j \in \Z)$ 
  and~\cite[Lemma~3.11]{Yamaguchi:asymptoticsRtorsion}.
  The Alexander polynomial $\Delta_{T(2, 2n+1)}(t)$ is given by
  $(t^{2n+1}+1)/(t+1)$.
  We have seen that $p_k$ is a divisor of $p$ in Corollary~\ref{cor:order_mu}.
  Since $\mathrm{gcd}(p, 2n+1)=1$,
  we can see that $\mathrm{gcd}(2p_k, 2n+1)=1$.
  From this, the denominator coincides with the numerator in
  the product of $|\Delta_{T(2, 2n+1)}(\xi^{2i-1}_k)|$
  except for $i=(p_k+1)/2$,
  \ie we have that 
  \[
  \prod_{\substack{1 \leq i \leq p_k, \\ i\not=(p_k+1)/2}}
  |\Delta_{T(2, 2n+1)}(\xi^{2i-1}_k)|
  =1.
  \]
  The right hand side of~\eqref{eqn:1st_term} turns into
  $(\log |\Delta_{T(2, 2n+1)}(-1)|) / p_k$.
  Hence the left hand side of~\eqref{eqn:main_result} turns out to be
  $(\log|\Delta_{T(2, 2n+1)}(-1)| - \log 2)/p_k$.
\end{proof}
It follows from Proposition~\ref{prop:set_eigenvalues} 
that the integer $p_k$, which gives the order of $\bar\rho(\mu)$ by $2p_k$,
runs over all divisors of $|\Delta_{K_n}(-1)|$ except for $1$.
\begin{corollary}
  \label{cor:set_limits}
  The set of the limits of the leading coefficients is given by
  \begin{equation}
    \label{eqn:set_limits}
  \left\{\left.
  \frac{1}{p_k}(\log |\Delta_{T(2, 2n+1)}(-1)| - \log 2) \,\right|\,
  p_k > 1, 
  \hbox{$p_k$ is a divisor of $|\Delta_{K_n}(-1)|$}\right\}.
  \end{equation}
  In particular, the minimum in the set~\eqref{eqn:set_limits} is 
  given by $(\log |\Delta_{T(2, 2n+1)}(-1)| - \log 2) / |\Delta_{K_n}(-1)| =
  (\log|2n+1|-\log2)/|4n+1|$.
\end{corollary}

\section*{Acknowledgments}
The first author was partially supported by a grant from the Simons Foundation (\#354595 to AT).
The second author was supported by JSPS KAKENHI Grant Number $26800030$. 


\newcommand{\noop}[1]{}
\providecommand{\bysame}{\leavevmode\hbox to3em{\hrulefill}\thinspace}
\providecommand{\MR}{\relax\ifhmode\unskip\space\fi MR }
\providecommand{\MRhref}[2]{%
  \href{http://www.ams.org/mathscinet-getitem?mr=#1}{#2}
}
\providecommand{\href}[2]{#2}

\end{document}